\documentclass[12]{amsart}

\usepackage{amssymb,amsmath,latexsym,amsthm}
\usepackage{MnSymbol}
\usepackage{graphicx}
\usepackage{fontenc}%
\usepackage{mathrsfs}
\usepackage{color}

\usepackage{lineno}

\usepackage[verbose]{wrapfig}

\usepackage{subfigure}
\renewcommand{\thesubfigure}{\thefigure.\arabic{subfigure}}
\makeatletter
\renewcommand{\p@subfigure}{}
\renewcommand{\@thesubfigure}{\thesubfigure:\hskip\subfiglabelskip}
\makeatother

\usepackage{pstricks,pst-text,pst-grad,pst-node,pst-tree,pst-plot,pst-barcode}
\usepackage{pstricks-add}
\usepackage{pst-solides3d}
\usepackage[tiling]{pst-fill}

\usepackage[english]{babel} 
\usepackage{pstricks} 
\usepackage{pstricks-add} 
\usepackage{paralist}

\newtheorem{theorem}{Theorem}[section]
\newtheorem{definition}[theorem]{Definition}

\newtheorem{proposition}[theorem]{Proposition}
\newtheorem{example}[theorem]{Example}

\newtheorem{remark}[theorem]{Remark}
\newtheorem{conjecture}[theorem]{Conjecture}

\def\defineTColor#1#2{%
 \newpsstyle{#1}{%
  fillstyle=vlines,hatchcolor=#2,
  hatchwidth=0.1\pslinewidth,
  hatchsep=1\pslinewidth}%
  }
\defineTColor{Tgray}{gray}
\defineTColor{Tgreen}{green}
\defineTColor{Tyellow}{yellow}
\defineTColor{Tred}{red} 
\defineTColor{Tblue}{blue} 
\defineTColor{Tmagenta}{magenta}
\defineTColor{Torange}{orange}
\defineTColor{Twhite}{white}
\defineTColor{Tgray}{lgrey}
\defineTColor{Tbgreen}{brightgreen}

\newcommand{\dnear}{\delta_{\Phi}}
\newcommand{\dcap}{\mathop{\cap}\limits_{\Phi}}

\newcommand{\dfar}{{\not\delta}_{\Phi}}

\newcommand{\snd}{\mathop{\delta_{_{\Phi}}}\limits^{\doublewedge}} 

\newcommand{\cl}{\mbox{cl}}

\begin{document}

\title[Descriptive Proximities I ]{Descriptive Proximities I:\\ Properties and interplay between \\ classical proximities and overlap}

\author[A. Di Concilio]{A. Di Concilio$^{\alpha}$}
\author[C. Guadagni]{C. Guadagni$^{\alpha}$}
\address{\llap{$^{\alpha}$\,}
Department of Mathematics, University of Salerno, via Giovanni Paolo II 132, 84084 Fisciano, Salerno , Italy}
\email{diconci@unisa.it,cguadagni@unisa.it}
\author[J.F. Peters]{J.F. Peters$^{\beta}$}
\address{\llap{$^{\beta}$\,}Computational Intelligence Laboratory,
University of Manitoba, WPG, MB, R3T 5V6, Canada and
Department of Mathematics, Faculty of Arts and Sciences, Ad\i yaman University, 02040 Ad\i yaman, Turkey}
\email{James.Peters3@umanitoba.ca}
\author[S. Ramanna]{S. Ramanna$^{\gamma}$}
\address{\llap{$^{\gamma}$\,}Applied Computer Science, University of Winnipeg, MB R3B 2E9, Canada}
\email{s.ramanna@uwinnipeg.ca}
\thanks{The research has been supported by the Natural Sciences \&
Engineering Research Council of Canada (NSERC) discovery grants 185986, 194376 
and Instituto Nazionale di Alta Matematica (INdAM) Francesco Severi, Gruppo Nazionale per le Strutture Algebriche, Geometriche e Loro Applicazioni grant 9 920160 000362, n.prot U 2016/000036.}


\date{}

\dedicatory{Dedicated to the Memory of Som Naimpally}

\begin{abstract}
The theory of descriptive nearness is usually adopted when dealing with sets that share some common properties even when the sets are not spatially close, {\em i.e.}, the sets have no members in common.  Set description results from the use of probe functions to define feature vectors that describe a set and the nearness of sets is given by their proximities.  A probe on a non-empty set $X$ is a real-valued function $\Phi: X \rightarrow \mathbb{R}^n$, where $\Phi(x)= (\phi_1(x),.., \phi_n(x))$. We establish a connection between relations on an object space $X$ and relations on the feature space $\Phi(X).$  Having  as starting point the Peters proximity, two sets are \emph{descriptively near}, if and only if their descriptions intersect.  In this paper, we construct a theoretical approach to a more visual form of proximity, namely, descriptive proximity, which has a broad spectrum of applications. We organize descriptive proximities on two different levels: weaker or stronger than the Peters proximity. We analyze the properties and interplay between descriptions on one side and classical proximities and overlap relations on the other side.
\end{abstract}


\maketitle
 
\section{Introduction}
This article carries forward recent work on proximities~\cite{Peters2015VoronoiAMSJ,Peters2015visibility,PetersGuadagni2015stronglyNear,PetersGuadagniTop,Guadagni2015}.  

\begin{figure}[!ht]
\centering
\subfigure[\footnotesize\bf Very Near Colour Sets]
 {\label{fig:veryNear}\includegraphics[width=20mm]{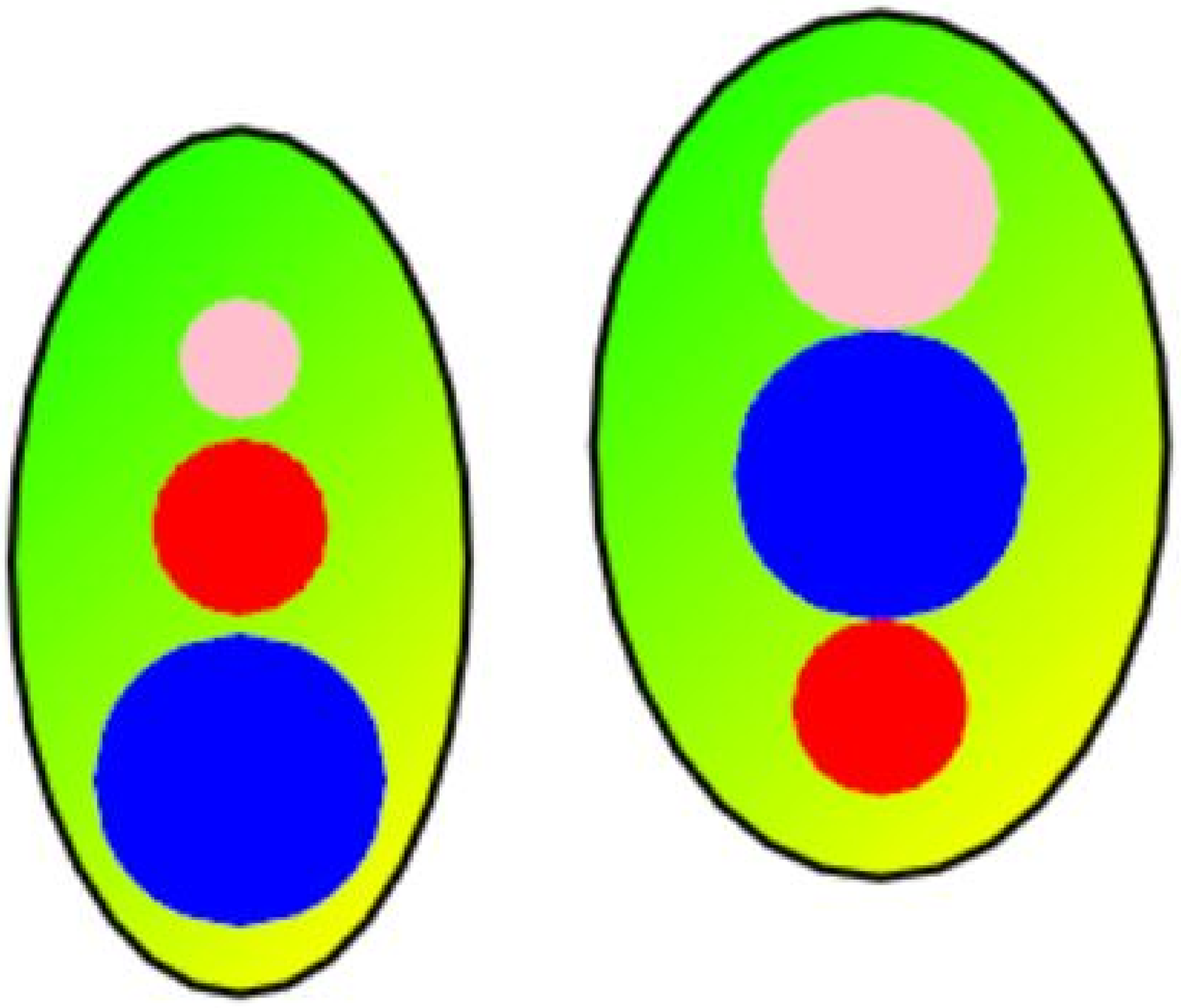}}\hfil
\subfigure[\footnotesize\bf Min. Near Colour Sets]
 {\label{fig:minNear}\includegraphics[width=20mm]{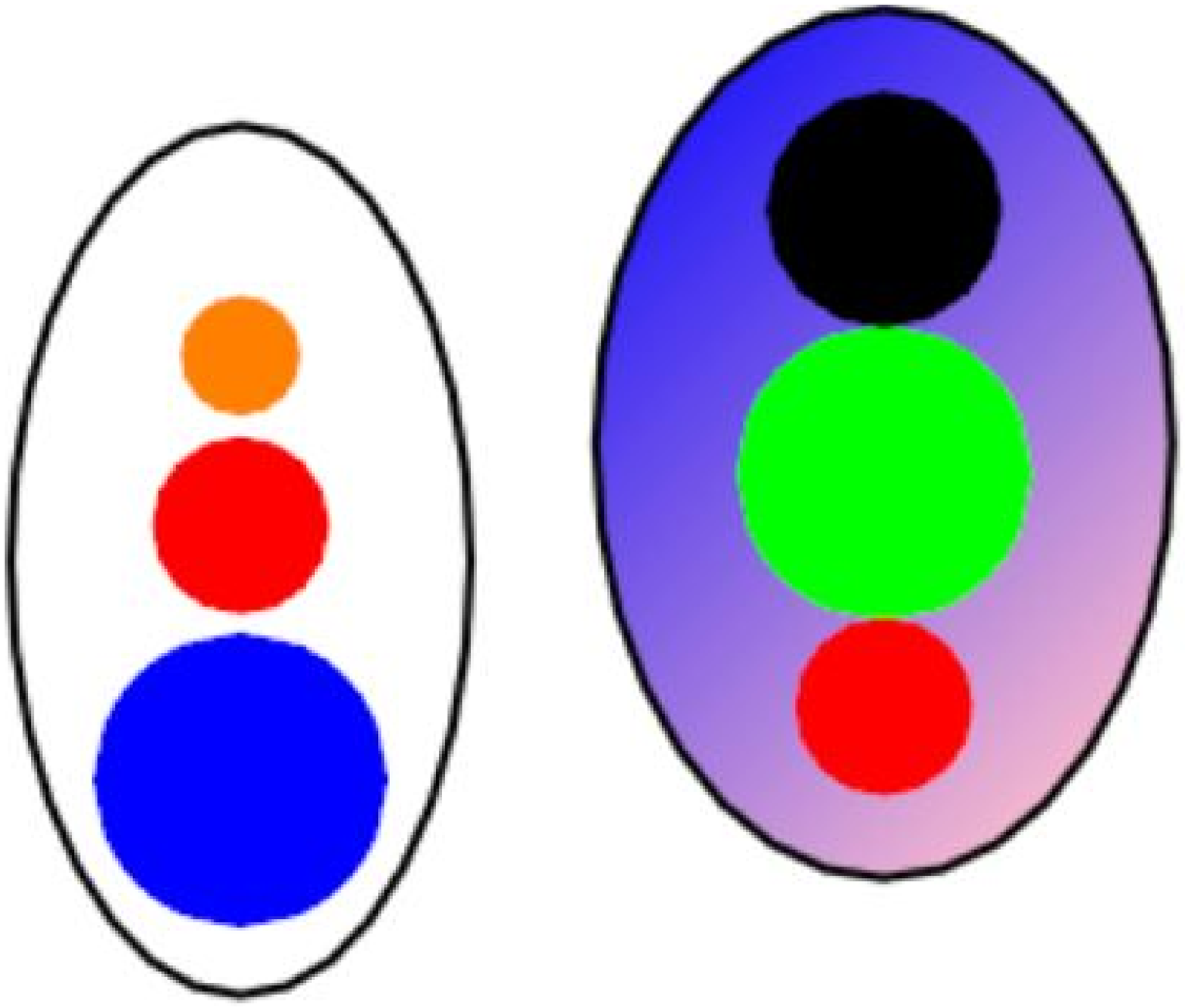}}\hfil
\subfigure[\footnotesize\bf Very Near Grey Sets]
 {\label{fig:veryNearGrey}\includegraphics[width=20mm]{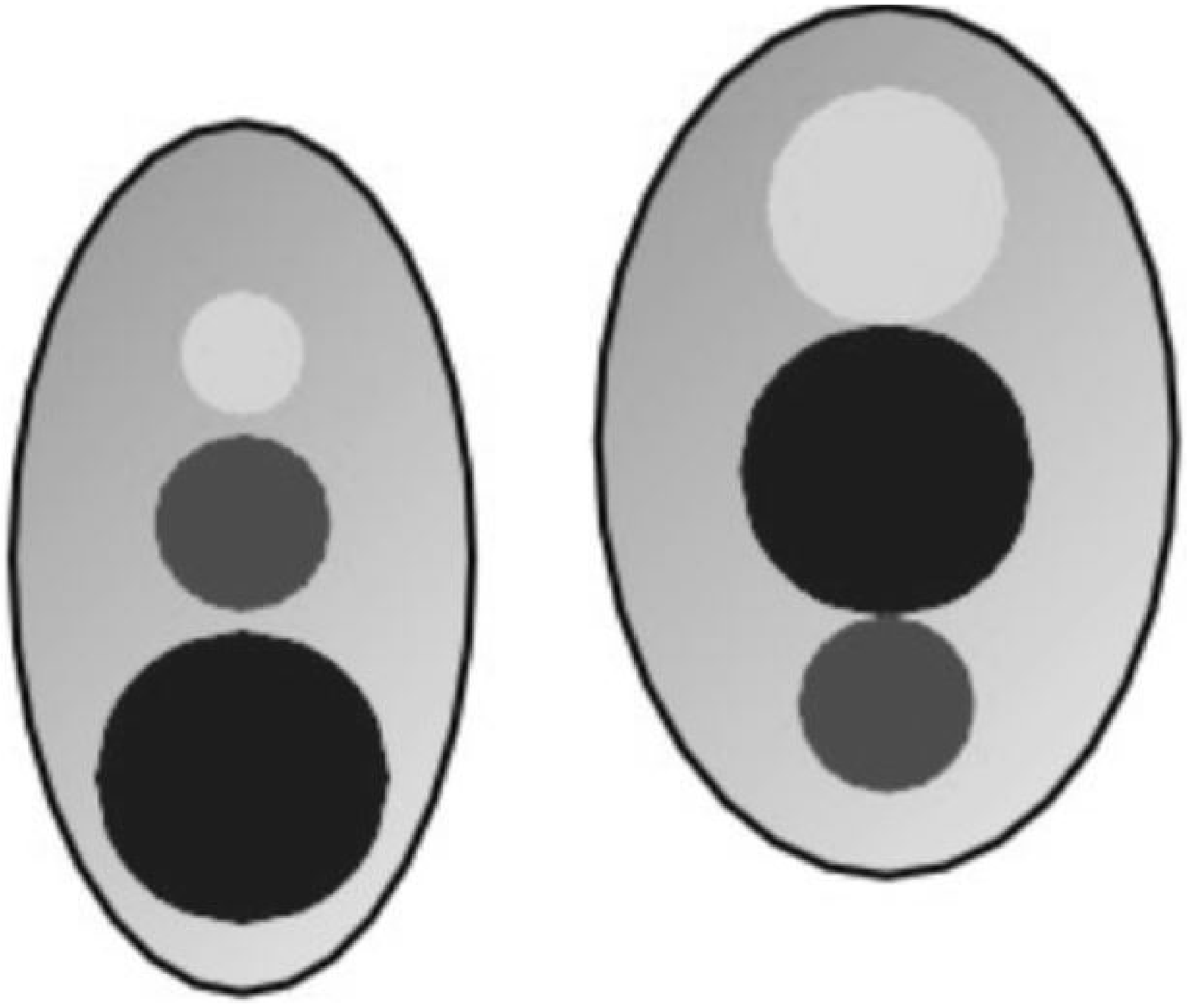}}\hfil
\subfigure[\footnotesize\bf Min. Near Grey Sets]
 {\label{fig:minNearGrey}\includegraphics[width=20mm]{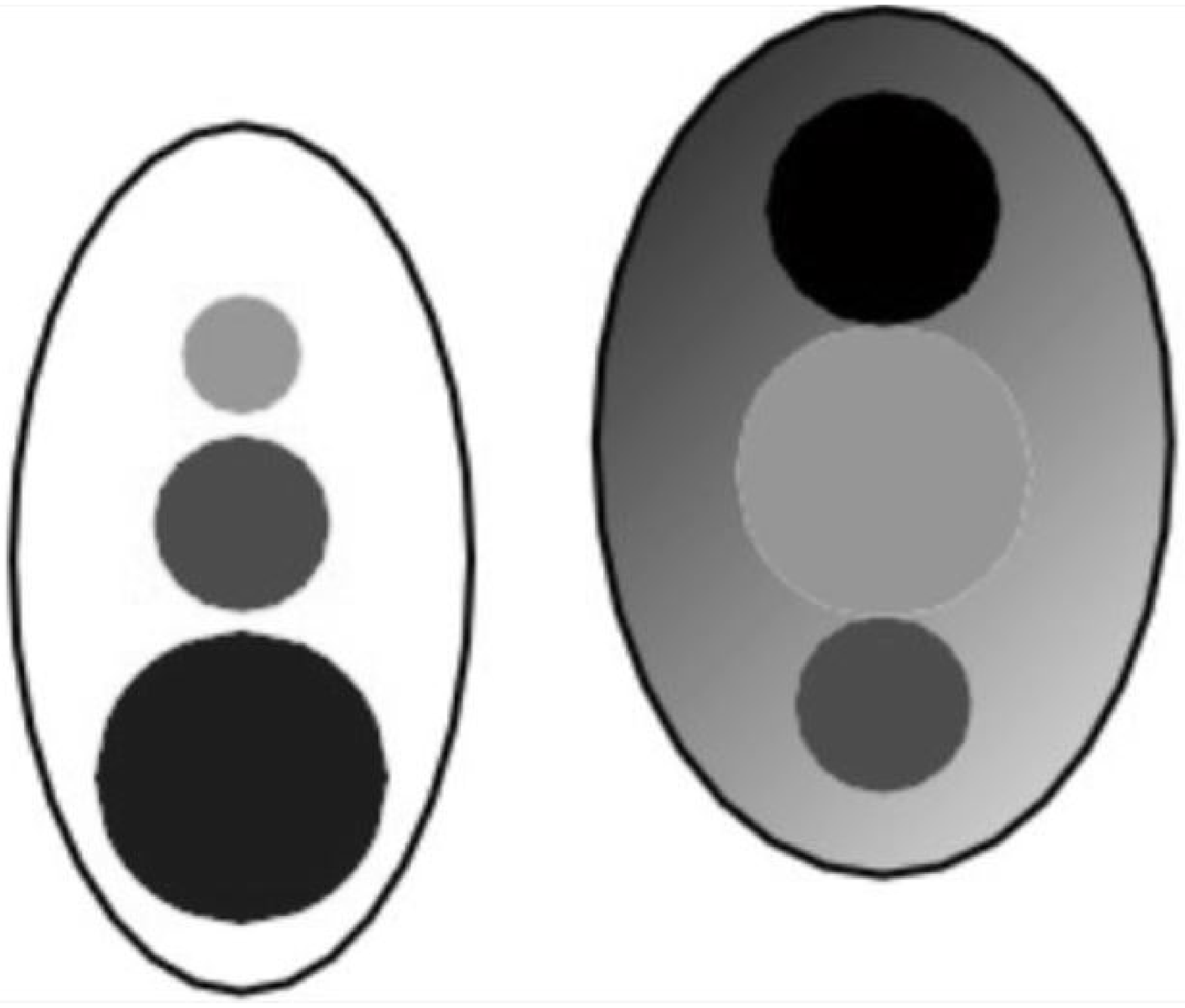}}
\caption[]{Descriptively near sets via colour or greyscale intensity}
\label{fig:nearSets}
\end{figure}
$\mbox{}$\\
\vspace{2mm}

Pivotal in this paper is the notion of a probe used to represent descriptions and proximities. 
A \emph{probe} on a non-empty set $X$ is real-valued function $\Phi: X \rightarrow \mathbb{R}^n$, where $\Phi(x)= (\phi_1(x),.., \phi_n(x))$ and each $\phi_i$ represents the measurement of a particular feature of an object $x\in X$~\cite{Peters2007AMS} (see also~\cite{Pavel1993})  . We establish a connection between relations on the object space $X$ and relations on  the feature space $\Phi(X).$  Usually, probe functions describe or codify physical features and act like "sensors" in extracting characteristic feature values from the objects.
The theory of descriptive nearness \cite{Peters2014} is usually adopted when dealing with subsets that share some common properties even though the subsets are not spatially close.

Each pair of ovals in Fig.~\ref{fig:veryNear} and Fig.~\ref{fig:minNear} contain circular-shaped coloured segments.  Each segment in the ovals corresponds to an equivalence class, where all pixels in the class have matching descriptions, {\em i.e.}, pixels with matching colours. 
For the ovals in Fig.~\ref{fig:veryNear} and Fig.~\ref{fig:minNear}, we observe that the sets are not spatially \emph{near,} but they can be considered near viewed in terms of colour intensities.  Again, for example, the ovals in Fig.~\ref{fig:veryNearGrey} and Fig.~\ref{fig:minNearGrey} contain segments that correspond to equivalence classes containing pixels with matching greyscale intensities.
The ovals in  Fig.~\ref{fig:veryNearGrey} and Fig.~\ref{fig:minNearGrey}
are descriptively near sets, since the equivalence classes contain matching greylevels.  Moreover,  we can also tell if they are more or less \emph{near.} In the sequel, we will express these ideas of resemblance in mathematical terms.

We talk about \emph{non-abstract points} when points have locations and features that can be measured.  The description-based theory is particularly relevant when we want to focus on some distinguishing characteristics of sets of non-abstract points. For example, if we take a picture element $x$ in a digital image, we can consider graylevel intensity or colour of $x$.
In general, we   define as a  \emph{probe}  an $n$ real valued  function $\Phi: X \rightarrow \mathbb{R}^n$, where $\Phi(x)= (\phi_1(x),.., \phi_n(x))$ and each $\phi_i$ represents the measurement of a particular feature. So, $\Phi(x)$ is a feature vector containing numbers representing feature values extracted from $x.$ And
$\Phi(x)$ is also called \emph{description} or \emph{codification} of $x$.  
Of course, nearness or apartness depends essentially on the selected features that are compared.

J.F. Peters~\cite[\S 1.19]{Peters2014} made the first fusion of description with proximity by introducing the notion of descriptive intersection of two sets:

\[ A \dcap B = \{x \in A \cup B: \Phi(x) \in \Phi(A), \ \Phi(x) \in \Phi(B) \}\]
and by declaring two sets  \emph{descriptively near}, if and only if their descriptive intersection is non empty or equivalently, if and only if their descriptions intersect. That is the first step in passing from the classical spatial proximity to the more visual descriptive proximity. The new point of view is a really different approach to proximity which has a broad spectrum of applications. The Peters proximity, which we will denote as $\pi_\phi,$ is the $\Phi-$pullback of the set-intersection.   By replacing the set-intersection with the descriptive intersection, we construct a theoretical approach  to the more visual form of proximity, namely, descriptive proximity (denoted by $\dnear$). We organize descriptive proximities in two different levels: weaker $  (A \delta{_\Phi} B \ \Rightarrow \   A \dcap B \neq \emptyset)$  or stronger  $( A \dcap B \neq \emptyset  \Rightarrow  A \delta{_\Phi} B) $ than the Peters proximity. In both cases, we find a natural underlying topology.  
That is, descriptive intersection
can be analyzed from the following two different perspectives: as the finest
 classical proximity, the discrete proximity, but also as the weakest overlapping relation, we exhibit significant examples of descriptive proximities weaker than the Peters proximity by following two different options: the proximal approach and the overlapping approach.

\subsection{Background of classical proximities}  
\

\noindent We draw our reference from  Naimpally-Di Concilio~\cite{DiConcilio2009,Naimpally1970} and are essentially interested in the simplest example of proximities, namely, Lodato proximities~\cite{Lodato1962,Lodato1964,Lodato1966} which guarantee the existence of a natural underlying topology.

\begin{definition}[Lodato]
Let $X$ be a nonempty set. A \textit{Lodato proximity $\delta$} is a relation on $\mathscr{P}(X)$, the collection of all subsets of $X,$ which satisfies the following properties for all subsets $A, B, C $ of $X$ :
\begin{description}
\item[$P_0)$] $A\ \delta\ B \Rightarrow A \neq \emptyset $ and $B \neq \emptyset $
\item[$P_1)$] $A\ \delta\ B \Leftrightarrow B\ \delta\ A$
\item[$P_2)$] $A \cap B \neq \emptyset \Rightarrow  A\ \delta\ B$
\item[$P_3)$] $A\ \delta\ (B \cup C) \Leftrightarrow A\ \delta\ B $ or $A\ \delta\ C$
\item[$P_4)$] $A\ \delta\ B$ and $\{b\}\ \delta\ C$ for each $b \in B \ \Rightarrow A\ \delta\ C$
\end{description}
Further $\delta$ is \textit{separated }, if 
\begin{itemize}
\item[$P_5)$] $\{x\}\ \delta\ \{y\} \Rightarrow x = y$.
\end{itemize}
\end{definition}

\noindent When we write $A\ \delta\ B$, we read "$A$ is near to $B$", while when we write $A \not \delta B$ we read "$A$ is far from $B$".  A relation  $\delta$ which  satisfies  only  $P_0)-P_3)$ is called a \emph{\u Cech}~\cite{Cech1966} or \emph{basic proximity}.
  
\vspace{2mm}

\noindent With any basic  proximity  one can associate a closure operator, $ \bf{\cl_\delta},$ by defining as closure of  any subset $A$ of $X:$
\[
\mbox{cl}_\delta A = \{ x \in X: \{x\} \ \delta\ A\}.
\]

\begin{definition}
An \emph{EF-proximity}~\cite{Efremovich1951,Efremovich1952} is a relation on $\mathscr{P}(X)$ which satisfies $P_0)$ through $P_3)$ and in addition the property:
\[ (EF) \ \ A \not\delta B \Rightarrow \exists E \subset X \hbox{ such that } A \not\delta E \hbox{ and } X\setminus E \not\delta B \]

which can be formulated  equivalently as:
\[ (EF1) \ \ A \not\delta B \Rightarrow \exists C, \ D  \subset X, \ C \cup D = X \hbox{ such that } A \not\delta C \hbox{ and } D \not\delta B .\]
\end{definition}
 
\noindent Since the EF-property is  stronger than the Lodato property, every EF-proximity is indeed a Lodato proximity.

\vspace{2mm}

\noindent The following remarkable properties reveals the potentialities of Lodato proximity. When $\delta$ is a Lodato proximity, then:

\begin{compactenum}[{\bf Property}.1] 
\item  The associated closure operator $\bf{\cl_\delta}$ is a Kuratowski operator~\cite{Kuratowski1958,Kuratowski1961}. Hence, every Lodato proximity space $(X, \delta)$ determines  an associated topology $\tau(\delta)$ whose closed sets are just the subsets which agree with their own closures.
\item  Furthermore, for each subsets  $A, B$ :
\begin{center} $A \ \delta \ B  \ \Longleftrightarrow \ \cl_{\delta} A \ \delta \ \cl_{\delta} B.$
\end{center}
\end{compactenum}

If $(X, \tau )$ is a topological space, we say that it admits a compatible Lodato proximity if there is a Lodato proximity $\delta$ on $X$ such that $\tau = \tau(\delta)$. A question arises when a topological space has a compatible Lodato proximity. This happens when the space satisfies the \emph{$R_0$-separation property}, i.e. $x \in \cl\{y\} \Leftrightarrow y \in \cl\{x\}$. In fact, every $R_0$ topological space $(X, \tau)$ admits as a compatible Lodato proximity $\delta_0$ given by: \

\label{def.fineLo}
\[ A \ \delta_0 \ B \Leftrightarrow \cl A \cap \cl B \neq \emptyset. \ \hbox{(Fine Lodato proximity~\cite{Naimpally1970})} \] 
\

On the other hand, a topological space has a compatible EF-proximity if and only if it is a completely regular topological space \cite{DiConcilio2009, Willard1970}.   Recall that a topological space is completely regular iff whenever $A$ is a closed set and $x \not\in A$, there is a continuous function 
$f: X \rightarrow [0,1]$ such that $f(x)=0$ and $f(A)=1$ \cite{Willard1970}.

\vspace{2mm}
$\bullet$ Any Lodato $T_1$ ( EF + $T_2$) proximity becomes \emph{spatial} by a $T_1$ ($T_2$) compactification procedure.\
\
\

\begin{figure}[!ht]
\centering
\includegraphics[width=25mm]{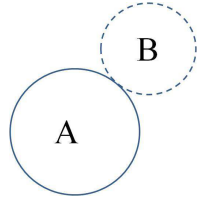}
\caption[]{Overlapping Sets:$A\ \delta\ B$}
\label{fig:fineLo}
\end{figure}

\subsection{ Examples}

\begin{example}
Consider $\mathbb{R}^2$ endowed with the Euclidean topology and the sets in Fig.~\ref{fig:fineLo}. $A$ is an open disk while $B$ is a closed disk. They are near in the fine Lodato proximity but they are far in the discrete proximity.
\qquad \textcolor{blue}{$\blacksquare$}
\end{example}

\begin{example}\label{ex:discreteProximity} {\bf Discrete Proximity on a Nonempty Set}.\\
Let $A,B\subset X$.  For a discrete proximity relation between $A$ and $B$, we have $A \ \delta \ B \Leftrightarrow \ A \cap B \neq \emptyset$.  This discrete proximity is a separated EF-proximity~[\S 2.1]\cite{DiConcilio2009}.
\qquad \textcolor{blue}{$\blacksquare$}
\end{example}

\noindent $\bullet$ From a spatial point of view, proximity appears as a generalization of the set-intersection.  The discrete proximity from Example~\ref{ex:discreteProximity} gives rise to a discrete topology.

\noindent $\bullet$ A pivotal EF-proximity is the \emph{metric proximity} $\delta_d$ associated with a metric space $(X,d)$ defined by considering the gap between two sets in a metric space 
( $d(A,B) = \inf \{d(a,b): a \in A, b \in B\}$ or $\infty$ if $A$ or $B$ is empty )
and by putting:
\[
A \ \delta_d \ B \Leftrightarrow  d(A, B) =0 .
\]

That is, $A$ and $B$  are $\delta_{d}-$near iff they  {\it either intersect or are asymptotic}: for each natural number $n$ there  is a  point  $a_n$ in $A $ and a point  $b_n$ in $B$ such that $d(a_{n},b_{n}) < \frac{1}{n}$.\\

\noindent $\bullet$ \emph{Fine Lodato proximity} $\delta_0$ on a topological space is defined as follows:

\[ A \ \delta_0 \ B \Leftrightarrow \cl A \cap \cl B \neq \emptyset.  \]

The proximity $\delta_0$ is the finest Lodato proximity compatible with a given topology.

\noindent $\bullet$ \emph{Functionally indistinguishable proximity} $ \delta_F$  on a completely regular space~\cite[\S 2.1,p.94]{DiConcilio2009}.

\begin{center}  { $ A \ \not\delta_{F}  \ B \Leftrightarrow  $ there is a continuous function  $ f: X \rightarrow [0,1] : \ f(A)=0 , \ f(B)= 1. $ }
\end{center}

\noindent The functionally indistinguishable  proximity on a completely regular space $X$  is  an EF-proximity, which is  further the finest EF-proximity compatible with $X.$  Moreover, $\delta_F$ coincides with the fine Lodato proximity if and only if $X$ is normal.

\begin{figure}[!ht]
 \centering
  \subfigure[EF relation]{\label{fig:EFsets}\includegraphics[width=25mm]{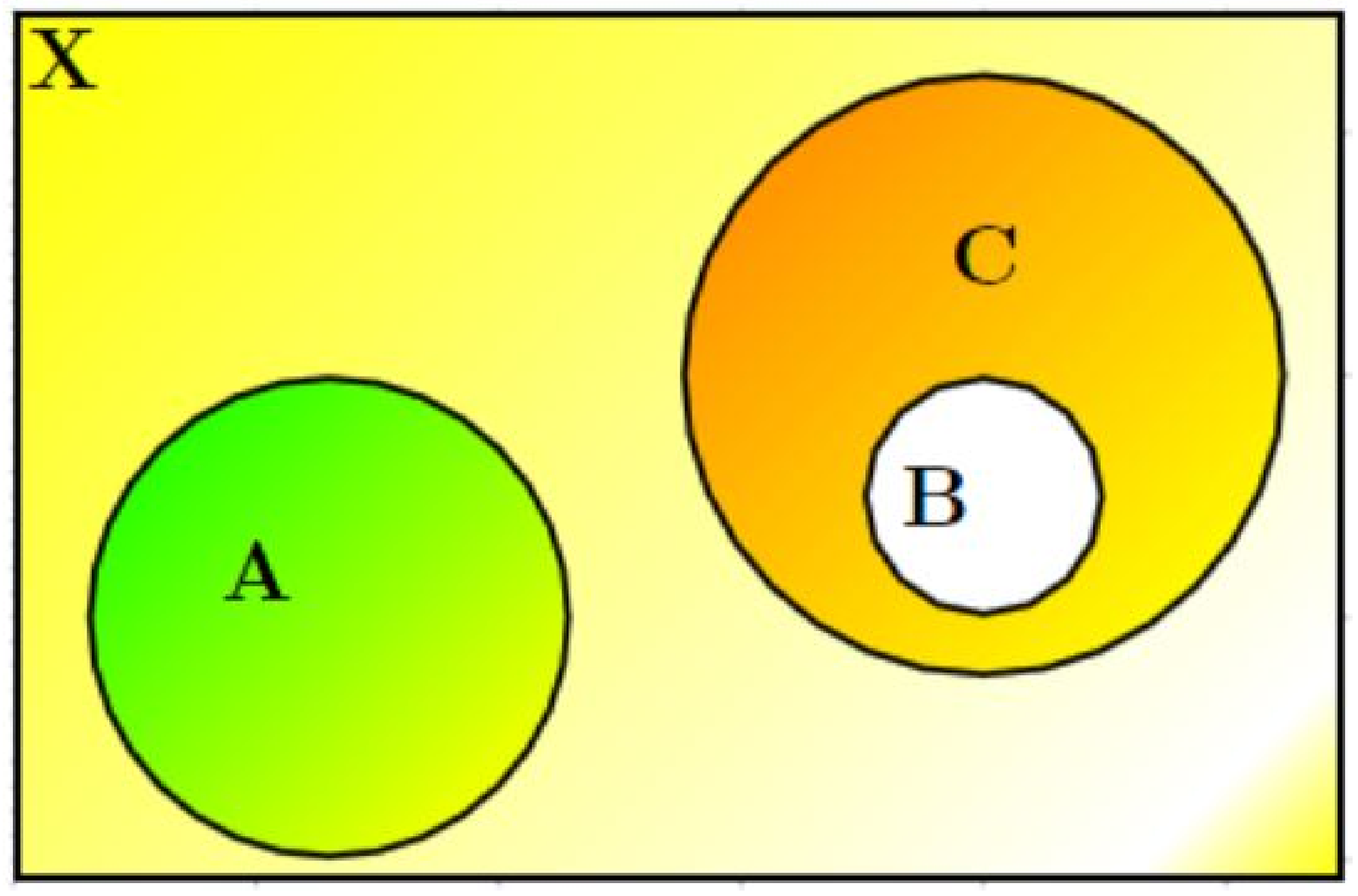}}\hfil
  \subfigure[EF display]{\label{fig:ef}\includegraphics[width=25mm]{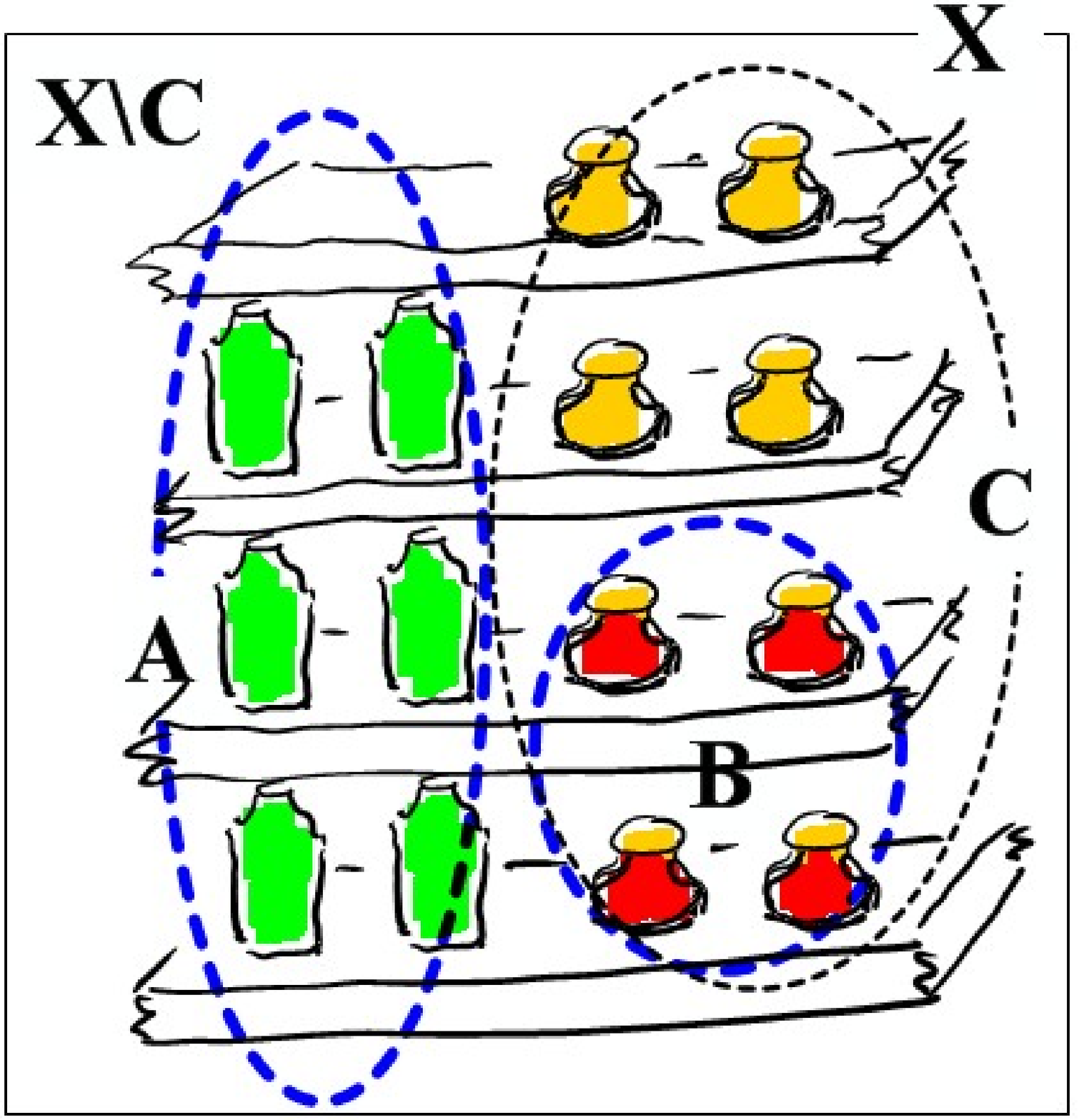}}\hfil
  \subfigure[Thai display]{\label{fig:thai}\includegraphics[width=25mm]{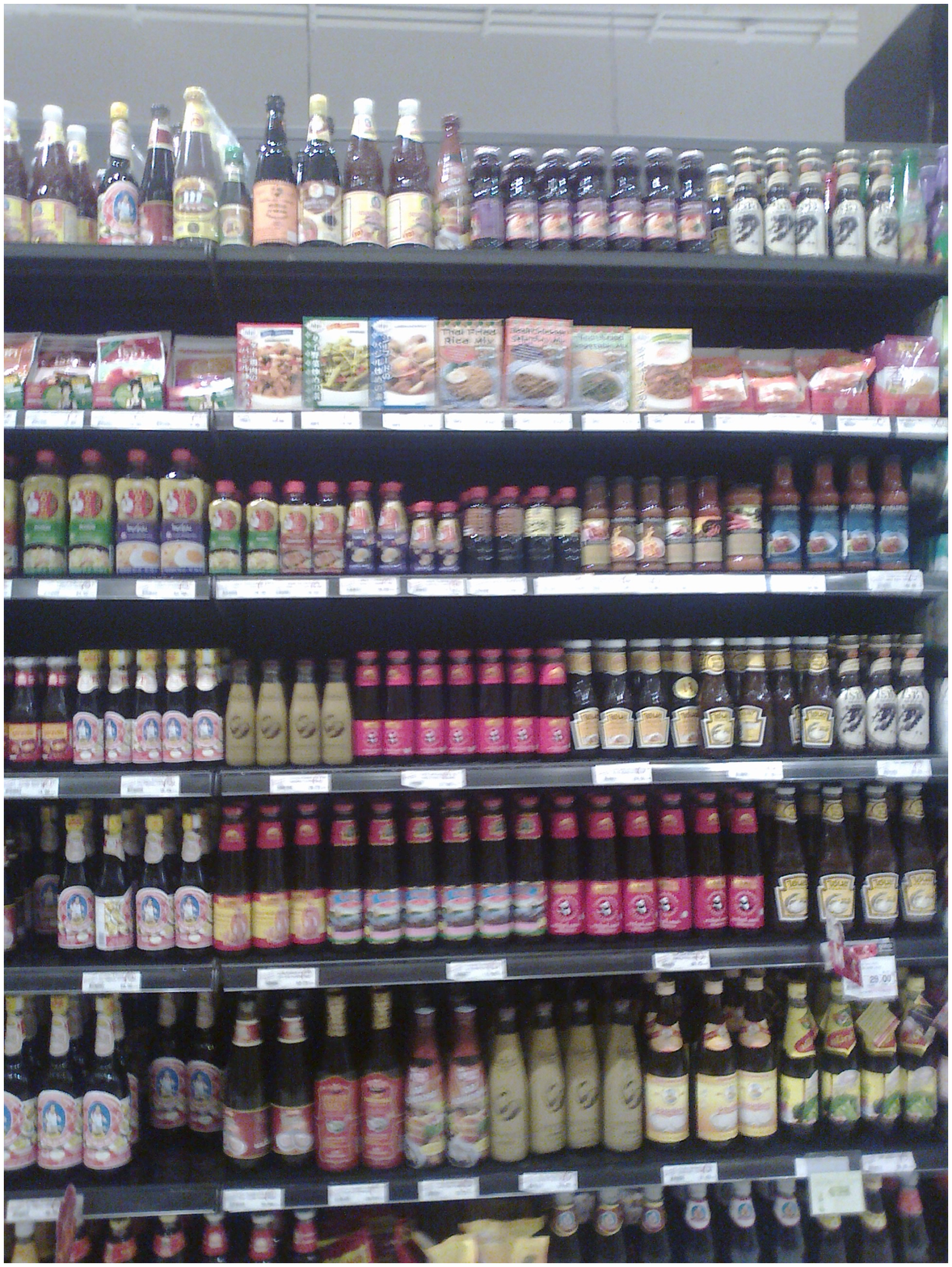}}
 \caption{Sample EF-Relationships}
 \label{fig:beh0} 
\end{figure}

\begin{example}{\bf Descriptive EF Proximity Relation~\cite{ Naimpally2013}}.\\
Let $A,C\subset X, B\subset C$ and let $C^c$ be the compliment of $C$.  A descriptive EF proximity (denoted by $\dfar$) has the following property:
\[
\boldsymbol{A}\ \dfar\ \boldsymbol{B}\Leftrightarrow \mathbf{A}\ \dfar\ \mathbf{C}\ \mbox{and}\ \mathbf{B}\ \dfar\ \mathbf{C^c}.
\]
A representation of this descriptive EF proximity relation is shown in Fig.~\ref{fig:EFsets}.  
The import of an EF-proximity relation is extended rather handily to visual displays of products in a supermarket (see, {\em e.g.}, Fig.~\ref{fig:thai}).  The sets of bottles that have an underlying EF-proximity to each other is shown conceptually in the sets in Fig.~\ref{fig:ef}.  The basic idea with this application of topology is to extend the normal practice in the vertical and horizontal arrangements of similar products with a consideration of the topological structure that results when remote sets are also taken into account, representing the relations between these remote sets with an EF-proximity.
\qquad \textcolor{blue}{$\blacksquare$}
\end{example}

\bigskip
\subsection{Strong inclusion}
\

Any proximity $\delta$ on $X$ induces a binary relation over the powerset $\mathscr{P}(X)$, usually denoted as $\ll_\delta$ and  named  the  {\it natural strong inclusion associated with } $\delta,$ by declaring that $ A$ is {\it strongly included} in $B, \ A \ll_{\delta} B,$ when $A$ is far from the complement of $B,\mbox{\emph{i.e.}}, A \not\delta X\setminus B$~\cite{DiConcilio2009}.
In terms of  strong inclusion associated with an EF-proximity  $\delta$, the \textit{Efremovi\v c property} for $\delta$ can be formulated  as the betweenness property:   \\

 \centerline {  (EF2) \  \   \  \  If $A \ll_{\delta} B,$  then there exists some $C$ such that $A \ll_{\delta} \  C \ll_{\delta} \ B$.} \

We conclude by emphasizing that a topological structure is based on the nearness between points and sets and a function between topological spaces is continuous provided it preserves nearness between points and sets, while a function between two proximity spaces is \emph{proximally continuous}, provided it preserves nearness between sets. Of course, any proximally continuous function is continuous with respect to the underlying topologies.
 
\section{Descriptive intersection and Peters proximity}
J.F. Peters made the first fusion of description with proximity, so passing from the classical spatial proximity to the recent more visual descriptive proximity which has a broad spectrum of applications~\cite{Peters2014,Peters2015AMSJmanifolds,Peters2012ams,PetersHettiarachchi2014manifolds,Peters2016Computational,PetersGuadagniTop}

\setlength{\intextsep}{0pt}

\begin{wrapfigure}[11]{R}{0.45\textwidth}
\begin{minipage}{5.5 cm}
\centering
\includegraphics[width=50mm]{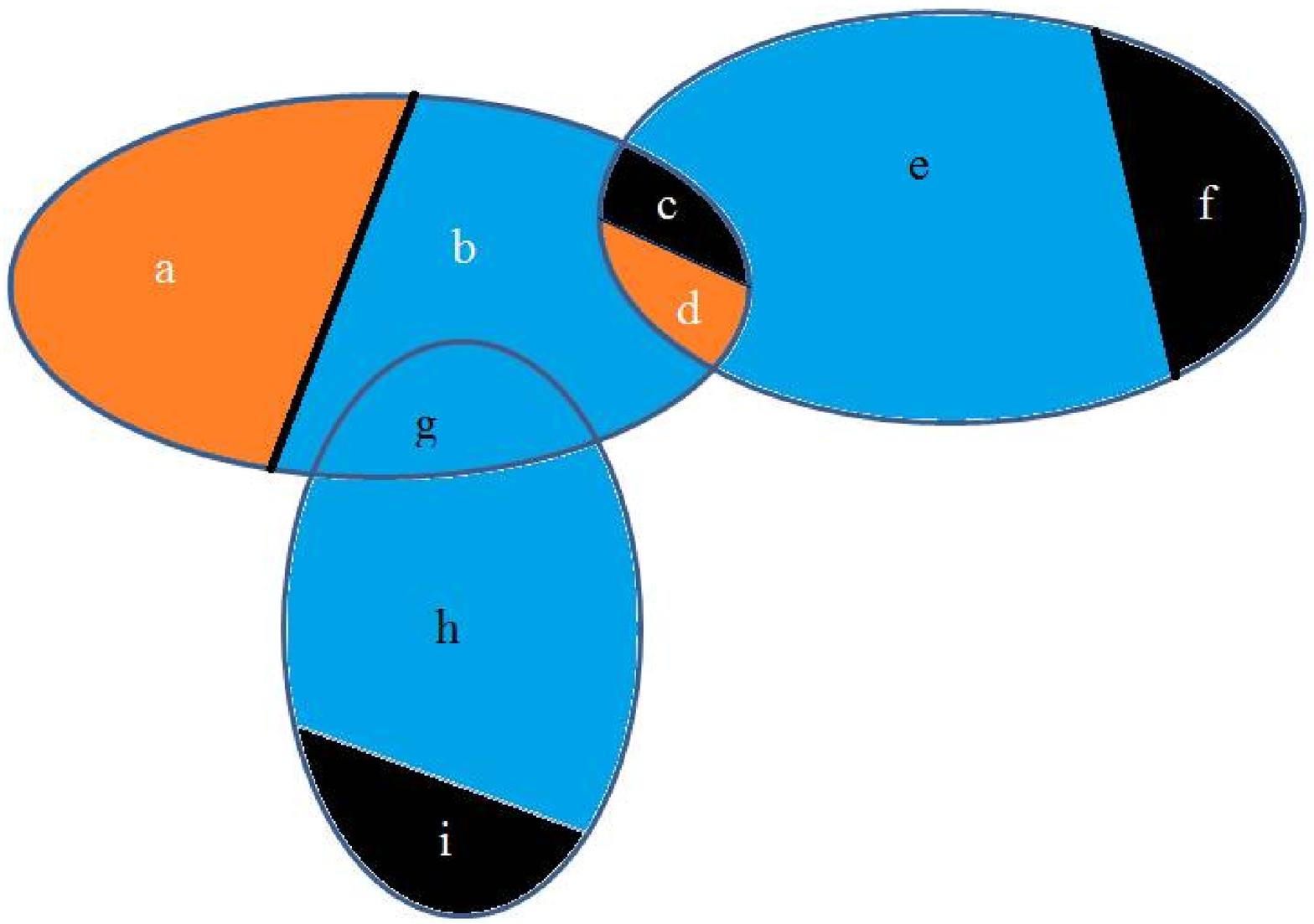}
\caption[]{$\mbox{}$\\ $\gamma_\Phi\neq\delta_{\Phi}$}
\label{fig.intersez}
\end{minipage}
\end{wrapfigure}

The starting idea is that two sets are near when the feature-values differences are so small so that they can be considered indistinguishable. He introduced the   notion  of descriptive intersection  which,  playing a similar role of set-intersection in the classical case, is crucial in our recent project to approach new forms of descriptive proximities. The mixture 
of description with proximity reveals an advantageous contamination.

The descriptive intersection of two sets $A, B$ is nonempty, provided there is at least one element in $A$ with a description that matches the description of at least one element in $B.$  The sets $A, B$ cannot share any point in common but
they can have a  nonempty descriptive intersection.

\begin{example}
Let $X$ be $\mathbb{R}^2$ and $\Phi: \mathbb{R}^2 \rightarrow \mathbb{R}^3$ be a probe that associate to each point its RGB-color. In Fig.~\ref{fig.intersez}, consider sets $A, B, C$ and their subsets $a, b, c, d, e, f, g ,h, i$. Observe that $\Phi(A) \cap \Phi(B)$ is given by colors black and red, so $ \Phi^{-1}(\Phi(A) \cap \Phi(B))= a \cup c \cup d \cup f \cup i$ and $A \dcap B = \Phi^{-1}(\Phi(A) \cap \Phi(B)) \cap (A \cup B)= a \cup c \cup d \cup f$. Then $A \dcap B \supseteq A \cap B = c \cup d$ and $A \dcap B \subseteq A \cup B$.
\qquad \textcolor{blue}{$\blacksquare$}
\end{example}

The first natural descriptive proximity, which  we decided to call Peters proximity and to denote as $\pi_\Phi$, declares two sets  \emph{descriptively near}   iff their descriptions intersect. Or in other words:
  
 \noindent Let $X$ be a non-empty set, $A$ and $B$ be subsets of $X,$ and $\Phi: X \rightarrow \mathbb{R}^n$ be a probe, then:
  $A \ \pi_{\Phi} \ B \Leftrightarrow \Phi(A) \cap \Phi(B) \neq \emptyset$.
\noindent namely, Peters proximity $\bf{\pi_\Phi}$, which is the $\Phi-$pull back of the discrete proximity.
 
\begin{theorem}   Peters proximity is  an  Efremovi\u c proximity, whose underlying topology is  $R_0$ and Alexandroff. Furthermore, $\bf{\pi_\Phi} $ is $T_0$, then $T_2,$ iff the probe $\Phi$  is injective.
\end{theorem}

Recall that a topological space has the
\emph{ Alexandroff property} iff  any intersection of open sets is in turn open~\cite{Alexandroff1935}.  
 
It is easily seen that we can rewrite the previous definition by using   $\Phi-$saturation of sets.

\begin{remark}
Recall that a set $A$ is called $\Phi-$saturated if and only if $\Phi^{-1}(\Phi(A))= A$.
\qquad \textcolor{blue}{$\blacksquare$}
\end{remark}

\begin{proposition}
Let $X$ be a non-empty set, $A$  be subset of $X$, and $\Phi: X \rightarrow \mathbb{R}^n$ be a probe. Then $A$ is closed in the topology induced by $\pi_\Phi$, $\tau(\pi_\Phi)$, if and only if it is $\Phi-$saturated. Moreover $\tau(\pi_\Phi)$ is disconnected.
\end{proposition}
\begin{proof}
The proof of the first part comes from the following equivalences :$x \in  Cl_{\pi_\Phi}(A) \Leftrightarrow x \ \pi_\Phi \ A  \Leftrightarrow  \ \Phi(x) \in \Phi(A) \Leftrightarrow \ x \ \in \Phi^{-1}(\Phi(A)). $ To see that $\tau(\pi_\Phi)$ is disconnected consider $\Phi^{-1}(\Phi(x))$. This is a closed set being equal to $Cl{\pi_\Phi}(x)$, but at the same time it is open because its complement is given by $\bigcup\{ \Phi^{-1}(\Phi(y)) : \Phi(y) \neq \Phi(x) \}$ and it is closed in its turn being $\Phi-$saturated.
\end{proof}

\bigskip

If we consider the relation on $X$ given by $ x \mathscr{R}_\Phi y \Leftrightarrow \Phi(x) = \Phi(y) $, then we have an equivalence relation whose classes are of type $[x] = \Phi^{-1}(\Phi(x))$, where $x \in X$. So two subsets of $X$, $A$ and $B$, are $\pi_\Phi-$near if and only if they intersect a same class of the partition induced by $\mathscr{R}_\Phi.$

\section{Descriptive proximities}
Peters proximity is a link between nearness or overlapping of descriptions in the codomain $\mathbb{R}^{n}$  with relations on pairs of subsets on the domain of codification.  But Peters proximity $\bf{\pi_\Phi}$ might be considered in some cases too strong or in some other ones too weak. So, by relaxing or stressing $\bf{\pi_\Phi},$ we obtain  general forms of descriptive proximities,  that can work better  than  it in particular settings.  Since, from a spatial point of view, classical  proximity is a generalization of the set-intersection, in our treatment we choose Peters proximity as the unique separation element between two different broad  classes of descriptive proximities. 
If we entrust  the descriptive intersection  with the same role of the set-intersection in the classical case we get the following two options: descriptive intersection versus descriptive proximity, {\em i.e.},
 
\begin{description}
\item[First option: weaker form]
\[A \dcap B \neq \emptyset \Rightarrow  A\ \dnear\ B 
\]
	
\item[Second option: stronger form]
\[
A\ \dnear\ B \Rightarrow A \dcap B \neq \emptyset.
\]
\end{description}
 
 \subsection{ Weaker form} \

 This is the case in which two sets having nonempty descriptive intersection are descriptively near:  $ A \dcap B \neq \emptyset \Rightarrow  A\ \dnear\ B. $  \

Let $X$ be a non-empty set, $A, \ B, \ C $  be subsets of $X$, and $\Phi: X \rightarrow \mathbb{R}^n$ be a probe. The relation $\dnear$ on $\mathscr{P}(X),$  the powerset of $X,$ is a  \emph{\v Cech 
$\Phi-$descriptive proximity} iff the following properties hold:

\begin{description}
\item[$D_0)$] $A\ \dnear\ B \Rightarrow A \neq \emptyset $ and $B \neq \emptyset $
\item[$D_1)$] $A\ \dnear\ B \Leftrightarrow B\ \dnear\ A$
\item[$D_2)$] $A \dcap B \neq \emptyset \Rightarrow  A\ \dnear\ B$
\item[$D_3)$] $A\ \dnear\ (B \cup C) \Leftrightarrow A\ \dnear\ B $ or $A\ \dnear\ C$

\end{description}

\noindent  If, additionally: \

  \[ {D_4):} \ \  A\ \dnear\ B \hbox{ and } \{ b \}\ \dnear\ C \hbox{ for each } b \in B \ \Rightarrow A\ \dnear\ C\]
 holds, then $\dnear $ is a  \emph{\ Lodato $\Phi-$descriptive proximity}~\cite[\S 4.15.2,p.155]{Peters2014}.\

\vspace{3mm}
\noindent Furthermore, if  the following  property holds:

$ \ \ A \not\delta_\Phi B \Rightarrow \exists E \subset X \hbox{ such that } A \not\delta_\Phi E \hbox{ and } X\setminus E \not\delta_{\Phi} B $, then $\dnear $ is an \emph{EF $\Phi-$descriptive proximity.} \

\ We explicitly observe that descriptive axioms  $D_0$ through $D_4$ are formally the same as  in the classical definition with  the set-intersection replaced by the $\Phi-$intersection.

 \subsection{The underlying topology} \
\

As in the classical case, for any descriptive proximity $\delta_\Phi$  and for  each subset $A$ in $X$ we define the  
\emph{$\Phi$-descriptive \  closure} of $A $ as :
\begin{center} $Cl_{\Phi}(A) =: \{ x \in X : x \ \delta_{\Phi}\  A \} $
\end{center}
\begin{theorem} The closure operator  $Cl_{\Phi}$ is a Kuratowski operator iff\ $\delta_\Phi$ is a Lodato  $\Phi-$descriptive proximity.
\end{theorem} 
\begin{proof}
Let $A,B,C\subset Cl_{\Phi}(D)$ and let $\delta_\Phi$ is a Lodato  $\Phi-$descriptive proximity.  The descriptive forms of P$_0$-P$_3$ of Lodato proximity for $\dnear$ are satisfied for $A,B,C$, if and only if $Cl_{\Phi}$ is a Kuratowski operator.
\end{proof}

\subsection{Examples} \

Peters proximity is the $\Phi-$pull back of the set-intersection. The set-intersection can be considered in two different aspects. It is the finest proximity on one side and the weakest overlap relation on the other side. So, to construct significant examples   of descriptive  proximities weaker than the Peters proximity we have two possible approaches: the proximal approach, which arises when  looking  at the the set-intersection as a proximity; the overlap approach, when looking at the set-intersection  as an overlap relation.  
 
\subsection{Proximity approach}
 Let $X$ be a nonempty set, $A$ and $B$ be subsets of $X, \  \Phi: X \rightarrow \mathbb{R}^n$ be a probe and $\delta$ be a proximity on $\mathbb{R}^n$. Then, if we define $\delta_\Phi$ as follows:

\begin{center}  $A \ \dnear \  B \Leftrightarrow \Phi(A) \ \delta \ \Phi(B) $ 
\end{center}
 we get a descriptive proximity. The descriptive proximity $\delta_\Phi $ and the  standard proximity $\delta$  are very close to each other absorbing  and transferring  their own   similar properties to the other. 
 \begin{theorem}
The proximity  $\delta$ is a  \u Cech, Lodato or  an EF-proximity  iff,  for each description $\Phi,    \  \dnear$ is  a  \u Cech, Lodato or an EF $\Phi-$ descriptive proximity.
\begin{proof}
We consider only classical EF proximity vs. EF $\Phi-$ descriptive proximity.  The equivalence between the two EF holds when the previous axioms hold.
\end{proof}
\end{theorem}

Observe that, given a proximity $\delta$ on $\mathbb{R}^n$, $\delta_\Phi$ is the coarsest proximity on $X$ for which the probe $\Phi$ is proximally continuous, i.e. $A \ \delta_\Phi \ B \Rightarrow \  \Phi(A) \ \delta \ \Phi(B)$~\cite[\S 1.7, p. 16]{Naimpally2013}.

Of course, the prototype is the Peters proximity when $\mathbb{R}^n$ is equipped with the discrete proximity. In  this case the $Cl_{\Phi}( A)$ is the  $\Phi-$preimage   of $\Phi(A).$
 
Another significant example is the fine Lodato descriptive proximity.

When $\mathbb{R}^n$ is equipped with the Euclidean topology, the finest Lodato proximity $\delta_0$ is an EF-proximity. The relative descriptive proximity  $\dnear^0,$ the fine Lodato descriptive proximity, is in its turn an EF-descriptive proximity.

  \emph{The  fine  Lodato  descriptive proximity}:
\[ A \  \dnear^0 \ B \Leftrightarrow  Cl_E(\Phi(A)) \cap  Cl_E(\Phi(B)) \neq \emptyset\]

\begin{conjecture} The fine Lodato descriptive proximity is the finest one  among all  "general" Lodato descriptive proximities as in the classical case.
\qquad \textcolor{blue}{$\blacksquare$}
\end{conjecture}

Based on the definition in \cite{Peters2016Computational}, we can also consider the descriptive closure of a set 
\[ Cl_{\Phi} A = \{x : \ x \dnear A \}= \{x: \Phi(x) \in  Cl_{E}(\Phi(A))\} \]

We prove now that we can re-write the fine descriptive proximity in terms of descriptive closures.

\

\begin{proposition}
 Let $X$ be a non-empty set, $A$ and $B$ be subsets of $X$, and $\Phi: X \rightarrow \mathbb{R}^n$ be a probe.
 \[ A \  \dnear^0 \ B \Leftrightarrow Cl(\Phi(A)) \cap  Cl(\Phi(B)) \neq \emptyset \Leftrightarrow Cl_{\Phi} A \cap  Cl_{\Phi} B \neq \emptyset.\]
\end{proposition}
\begin{proof}
$Cl(\Phi(A)) \cap \cl(\Phi(B)) \neq \emptyset \Leftrightarrow \exists y \in \Phi(X):  y \ \in  Cl(\Phi(A)) \cap  Cl(\Phi(B)) \Leftrightarrow y \ \delta \ \Phi(A)$ and $y \ \delta \ \Phi(B) \Leftrightarrow \exists x \in X: y= \Phi(x), \ \Phi(x) \ \delta \ \Phi(A)$ and $\Phi(x) \ \delta \ \Phi(B) \ \Leftrightarrow \exists x \in  Cl_{\Phi} A \cap  Cl_{\Phi} B \Leftrightarrow  Cl_{\Phi} A \cap  Cl_{\Phi} B \neq \emptyset.$
\end{proof}

\setlength{\intextsep}{0pt}

\begin{wrapfigure}[13]{R}{0.35\textwidth}
\begin{minipage}{4.2 cm}
\centering
\includegraphics[width=30mm]{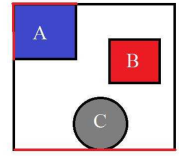}
\caption[]{$\mbox{}$\\ $\pi_\Phi \Rightarrow \beta_\Phi$}
\label{fig.pi-beta}
\end{minipage}
\end{wrapfigure} 

When requiring   $A \ \pi_\Phi \ B,$ we look at the match of the entire feature vectors on points of $A$ and $B$. But, it can be useful to consider a fixed part of the vector of feature values.  In this way descriptive nearness of sets can be established  on a partial match of descriptions. To achieve this result, we introduce:
\begin{definition}{$\left(\boldsymbol{\beta_{\Phi}}\right)$}.\\
Let $X$ be a non-empty set, $A$ and $B$ be subsets of $X$,\\ 
and $\Phi: X \rightarrow \mathbb{R}^n$ be a probe. We define 
\begin{center}
 $A \ \beta_\Phi \ B \Leftrightarrow \Phi_i(A) \cap \Phi_i(B) \neq \emptyset, \ \forall i = 1,...,n $
\end{center}
 \end{definition}
\
Further, by generalizing $\beta_\Phi$  by composing the probe $\Phi$ with the projection $\pi_m$:
\begin{center}
{\small $x \ \longrightarrow  \     (\pi_m \circ \Phi)(x)= (\phi_1(x),..., \phi_m(x)).$}
\end{center}
we have:
\begin{definition}{$\left(\boldsymbol{\eta_{\Phi}}\right)$}.\\
\begin{center}
$ A \ \eta_\Phi \ B \Leftrightarrow \pi_m(\Phi(A)) \cap \pi_m(\Phi(B)) \neq \emptyset. $
\end{center}
 \end{definition}

\begin{proposition}\label{prop:eta}
The relation $\eta_\Phi,$ then the relation $\beta_\Phi,$ is a $\Phi-$descriptive EF-proximity.
\qquad \textcolor{blue}{$\blacksquare$}
\end{proposition}

\begin{example}
For an illustration of Prop.~\ref{prop:eta}, see Fig.~\ref{fig.pi-beta}.
\end{example}
 
\begin{remark}
The topology associated with $\beta_\Phi$ is defined by the Kuratowski operator $Cl_{\beta_\Phi}$:
\[ x \in Cl_{\beta_\Phi}(A) \Leftrightarrow x \in \bigcap_{i=1,..,n} \Phi_i^{-1}(\Phi_i(A))
\mbox{\qquad \textcolor{blue}{$\blacksquare$}}
\]
\end{remark}

A third kind of descriptive relation, but not a descriptive nearness, defined  by probes and intersection is given as follows.

\begin{definition}\label{def.gamma}{$\left(\boldsymbol{\gamma_{\Phi}}\right)$}.\\
 Let $X$ be a non-empty set, $A$ and $B$ be subsets of $X$, and $\Phi: X \rightarrow \mathbb{R}^n$ be a probe. We define
$ A\ \gamma_{\Phi}\ B \Leftrightarrow \ \exists i \in \{1,..,n\} : \ \Phi_i(A) \cap \Phi_i(B) \neq \emptyset $. 
 \end{definition}

\begin{figure}[!ht]
\centering
\includegraphics[width=35mm]{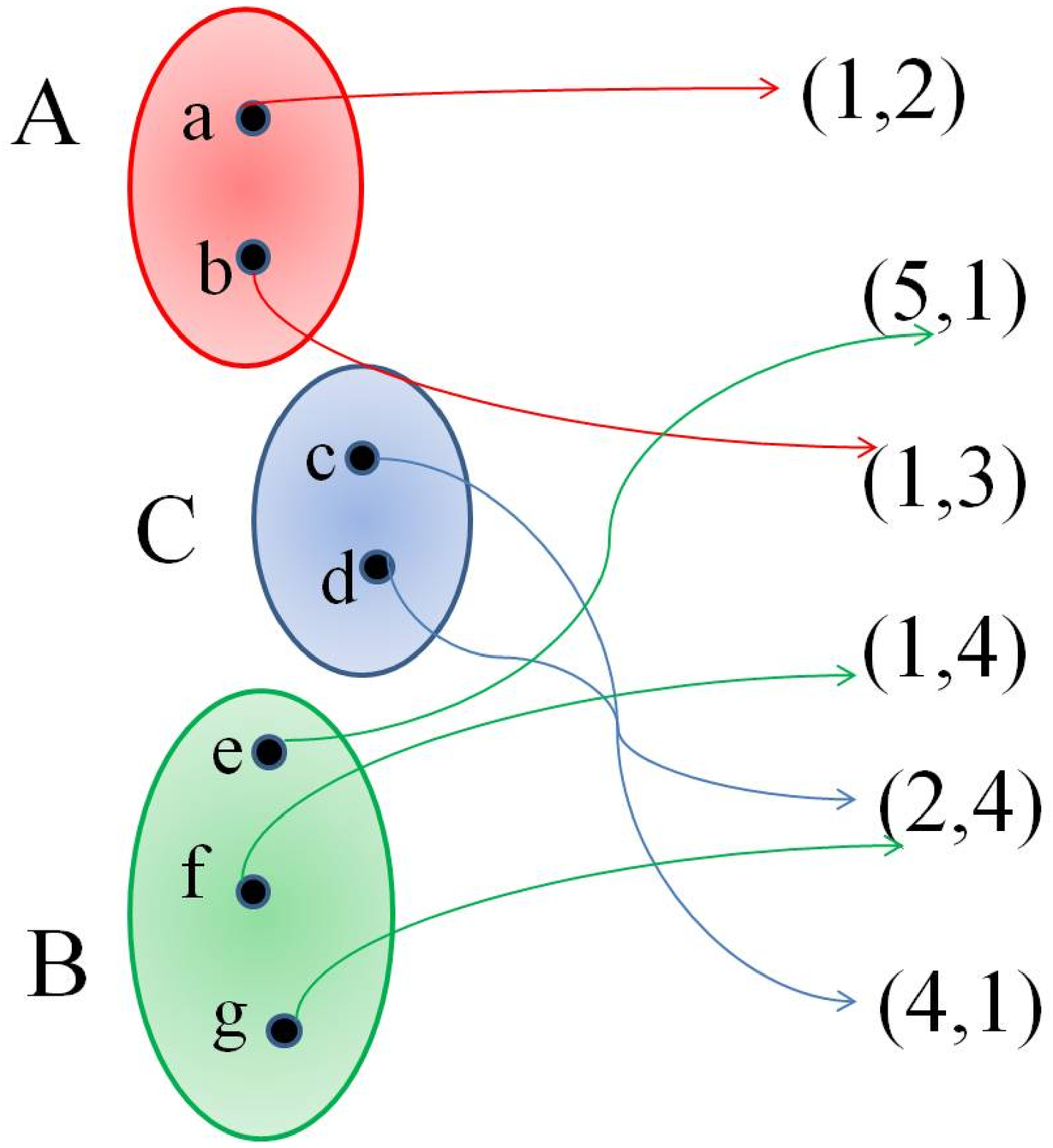}
\caption[]{$\gamma_\Phi\neq\delta_{\Phi}$}
\label{fig.gamma}
\end{figure}
 
 The relation $\gamma_\Phi$ is not a  descriptive proximity. We illustrate this by the following example based on Fig.~\ref{fig.gamma}.

\begin{example}
Let $A = \{a, b\}$, 
$C= \{c, d\}$, 
$B= \{e, f , g \}$. 
In this figure we have  
$\Phi_1(A)= \{1\}, \ \Phi_2(A)= \{2,3\}, \Phi_1(C)= \{2,4\}, \ \Phi_2(C)= \{4,1 \}, \ \Phi_1(B)= \{1, 2, 5 \}$. So $A \gamma_\Phi B$ because $\Phi_1(A) \cap \Phi_1(B) \neq \emptyset$, and for each $ x \in C$ $x  \ \gamma_\Phi \ C$. But $A \not\gamma_\Phi C$ because $\Phi_1(A) \cap \Phi_1(C) = \emptyset$ and $\Phi_2(A) \cap \Phi_2(C) = \emptyset$. In other words $\gamma_\Phi$ is not a Lodato proximity~\cite[\S 3.1, p. 72]{Naimpally2013}. 
\qquad \textcolor{blue}{$\blacksquare$}
\end{example}

\subsection{Overlapping approach}  \

Suppose that  for any subset $A$ of $X$ a specific enlargement, $e(\Phi(A)), \ $  of $\Phi(A)$  in $\mathbb{R}^n$ can be associated with $A$ and moreover, for any pair $A,B  \ :   e(\Phi(A)) \cup  e(\Phi(B)) =  e(\Phi(A) \cup  \Phi(B))$ (\emph{additivity}) and also $A \subseteq B  \ \Rightarrow  e(\Phi(A))  \subseteq   e(\Phi(B)),$ (\emph{extensionality})\cite{DiConcilio2013mcs}.  Then, if we put:

 $$ A \ \dnear \  B \  \hbox{ iff}  \  e(\Phi(A) \cap e(\Phi(B) \neq \emptyset$$

we have:
\begin{proposition} The relation  $\dnear \ $ is a 
$\Phi-$descriptive Lodato proximity.
\begin{proof} This result follows from the initial conditions. 
\end{proof}
\end{proposition}

When choosing as $\epsilon >0 $ as level of  approximation  and as enlargement for any subset of  $\mathbb{R}^n$ the $\epsilon-$enlargement, we have a peculiar case in the overlapping approach.
It is not possible to remove additivity or extensionality as the following  geometric example, related to the affine structure of $\mathbb{R}^n$, proves:
 
 \begin{center} 
$A \ \dnear \  B \Leftrightarrow  \hbox{conv}(\Phi(A)) \cap \hbox{conv}(\Phi(B)) \neq \emptyset, $
\end{center}
\noindent where conv$(\Phi(A))=$ minimal convex set containing $\Phi(A).$ 
The above  relation verifies the properties $D_0, D_1, D_2, D_4$ but only one way in $D_3.$

 \subsection{Second option: stronger form}

   This  is the case in which two sets descriptively near have a 
	nonempty descriptive intersection: $A\ \dnear\ B \Rightarrow A \dcap B \neq \emptyset. $ 
	\	

Let $X$ be a non-empty set, $A, \ B, \ C $  be subsets of $X$  and $\Phi: X \rightarrow \mathbb{R}^n$ be a probe.

The relation $\snd$ on $\mathscr{P}(X)$ is a $\Phi-$descriptive  Lodato  strong proximity\cite{PetersGuadagni2015stronglyNear} iff the following properties hold:

\begin{description}
\item[$(S_0)$] $A\ \snd\ B \Rightarrow A \neq \emptyset $ and $B \neq \emptyset $
\item[$(S_1)$] $A\ \snd\ B \Leftrightarrow B\ \snd\ A$
\item[$(S_2)$] $A \snd B \Rightarrow  A \dcap B  \neq \emptyset$
\item[$(S_3)$] $A\ \snd\ (B \cup C) \Leftrightarrow A\ \snd\ B $ or $A\ \snd\ C$ 
\item[$(S_4)$] $\ A\ \snd\ B \hbox{ and } \{ b \}\ \snd\ C \hbox{ for each } b \in B \ \Rightarrow A\ \snd\ C .$
\end{description}
\bigskip

 As an example, when we can distinguish a significant subset $S \subseteq \Phi(X),$ we can put:  $ A\ \snd\ B $ if and only if 
  $\Phi(A)$ shares some common point with $ \Phi(B)$
   belonging to $S$, and obtain a strong  $\Phi-$descriptive proximity.

\bibliographystyle{amsplain}
\bibliography{CPrefs}

\end{document}